\documentclass[11pt]{article}

\usepackage{amsmath}
\usepackage{amssymb}
\usepackage{amsthm} 
\usepackage{mathrsfs}
\usepackage[dvipdfmx]{graphicx,color}
\nopagebreak[3]

\newcommand{\newsection}[1]
{\section{#1}\setcounter{theorem}{0} \setcounter{equation}{0} \par\noindent}

\newtheorem{theorem}{Theorem}

\newtheorem{lemma}[theorem]{Lemma}

\newtheorem{proposition}[theorem]{Proposition}

\newcommand{\beq}{ \begin{equation} }
\newcommand{\eeq}{ \end{equation} }

\newcommand{\br}{{\mathbb R}}

\newcommand{\supp}{\mbox{\rm supp}\, }

\newcommand{\diag}{\mbox{\rm diag}\, }

\newcommand{\brn}{ { \mathbb{R}^n } }

\newcommand{\real}{ \,\mbox{\rm{Re}}\, }

\title{
The Cauchy problem for the Klein-Gordon equation 
\\
under the quartic potential in the de Sitter spacetime
}
\author
{
Makoto NAKAMURA
\thanks
{
Faculty of Science, Yamagata University, 
Kojirakawa-machi 1-4-12, Yamagata 990-8560, JAPAN.
E-mail: \texttt{nakamura@sci.kj.yamagata-u.ac.jp}
}
}
\date{}

\begin{document}

\maketitle

\begin{abstract}
The Cauchy problem for the Klein-Gordon equation under the quartic potential is considered in the de Sitter spacetime.
The existence of the global solution is shown based on the mechanism of the spontaneous symmetry breaking for the small positive Hubble constant.
The effects of the spatial expansion and contraction on the problem are considered.
\end{abstract}



\newsection{Introduction}
We consider the Cauchy problem for the Klein-Gordon equation under the quartic potential  
in the de Sitter spacetime.
The de Sitter spacetime is   
the solution of the Einstein equations 
with the cosmological constant in the vacuum 
under the cosmological principle.
We use the following convention.
Let $n\ge1$.
The Greek letters $\alpha, \beta,\gamma,\cdots$ run from $0$ to $n$, 
and the Latin letters $j,k,\ell, \cdots$ run from $1$ to $n$.
We use the Einstein rule for the sum of indices, 
namely, the sum is taken for the same upper and lower repeated indices, 
for example, 
${T^\alpha}_\alpha:=\sum_{\alpha=0}^n {T^\alpha}_\alpha$ 
and 
${T^j}_j:=\sum_{j=1}^n {T^j}_j$ for any tensor ${T^\alpha}_\beta$.
Let $H\in \br$ be the Hubble constant, 
$c>0$ be the speed of the light.
The de Sitter spacetime that we consider in this paper is the spacetime with the metric $\{g_{\alpha\beta}\}$ given by 
\begin{equation}
\label{Intro-RW}
-c^2(d\tau)^2
=
g_{\alpha\beta}dx^\alpha dx^\beta
:=
-c^2(dx^0)^2+
e^{2H x^0}
\sum_{j=1}^n (dx^j)^2,
\end{equation}
where we have put the spatial curvature as $0$,  
the variable $\tau$ denotes the proper time, 
$x^0=t$ is the time-variable 
(see e.g., 
\cite{Carroll-2004-Addison, DInverno-1992-Oxford}).
When $H=0$, the spacetime with \eqref{Intro-RW} reduces to the Minkowski spacetime.


For the imaginary number $i$ with $i^2=-1$,
let $m_\ast\in \br \cup i\br$ denote 
the real mass when $m_\ast\in \br$, 
the imaginary mass when $m_\ast\in i\br$.
Let $\hbar>0$ denote the reduced Planck constant. 
The Lagrangian density  $L$ of the Klein-Gordon field with a potential $V$ given by 
\beq
\label{Def-L-V}
L(\phi):=
-\frac{1}{2}\overline{\partial^\alpha \phi} \partial_\alpha \phi
-
V(\phi),
\ \ 
V(\phi):=
\frac{1}{2}
\left(\frac{m_\ast c}{\hbar}\right)^2|\phi|^2
+
\frac{\lambda}{p+1}|\phi|^{p+1},
\eeq
yields the semilinear Klein-Gordon equation 
\begin{equation}
\label{KG-Cauchy}
\left(
\frac{1}{c^{2}}\partial_t^2\phi
+\frac{nH}{c^2}\partial_t\phi
-e^{-2Ht}
\Delta \phi
+\left(\frac{m_\ast c}{\hbar}\right)^2 \phi
+
\lambda |\phi|^{p-1}\phi
\right)(t,x)
=0, 
\end{equation}
for $(t,x)\in [0,T)\times\br^n$, 
where 
$0< T\le\infty$,
$1<p<\infty$, 
$\lambda\in \br$, 
$\Delta$ denotes the Laplacian defined by $\Delta:=\sum_{j=1}^n (\partial/\partial x^j)^2$,
$\partial^\alpha:=g^{\alpha\beta}\partial_\beta$,
and 
the matrix $(g^{\alpha\beta})$ denotes the inverse matrix $(g_{\alpha\beta})$. 
The equation \eqref{KG-Cauchy} is rewritten as 
\beq
\label{Eq-u-GaugeVariant}
c^{-2} \partial_t^2u-e^{-2Ht}\Delta u
+
\left\{\left(\frac{m_\ast c}{\hbar}\right)^2-\left(\frac{nH}{2c}\right)^2\right\}
u
+
\lambda e^{-n(p-1)Ht/2} |u|^{p-1}u=0
\eeq
by the change of $\phi$ to $u:=e^{nHt/2}\phi$.

Let us consider the real mass $m_\ast\in \br$ 
on the Cauchy problem of \eqref{KG-Cauchy} for data 
$\phi_0(\cdot):=\phi(0,\cdot)$ and $\phi_1(\cdot):=\partial_t\phi(0,\cdot)$. 
Yagdjian \cite{Yagdjian-2012-JMAA} has shown 
small global solutions for \eqref{KG-Cauchy}, 
provided by that the norm of initial data $\|\phi_0\|_{H^s(\brn)}+\|\phi_1\|_{H^s(\brn)}$ is sufficiently small for $s>n/2\ge 1$  
(see also  \cite{Yagdjian-2013-Springer} for the system of the equations), 
where $H^s(\mathbb{R}^n)$ denotes the Sobolev space of order $s$. 
Galstian and Yagdjian 
\cite{Galstian-Yagdjian-2017-NA-RWA, 
Yagdjian-2019-JMP} 
have extended 
this result to the case of the Riemann metric space for each time slices.
In \cite{Nakamura-2014-JMAA}, the energy solutions 
for $\phi_0\in H^1(\brn)$ and $\phi_1\in L^2(\brn)$ 
have been shown, 
which was extended to the case of general Friedmann-Lema\^itre-Robertson-Walker spacetime 
in \cite{Galstian-Yagdjian-2015-NA-TMA}.
Baskin \cite{Baskin-2012-AHP} has shown small global solution for the equation 
$(\square_g+\lambda) \phi+f(\phi)=0$ 
when $f(\phi)$ is a type of $|\phi|^{p-1} \phi$,  
$p=1+4/(n-1)$, $\lambda>n^2/4$, $\phi_0\in H^1$ and $\phi_1\in L^2$, 
where $g$ denotes the metric of the asymptotic de Sitter spacetime 
and $\square_g$ denotes the d'Alembertian on $g$  
(see also \cite{Baskin-2010-ProcCMAANU} in the cases $\lambda=(n^2-1)/4$, $p=5$ with $n=3$, $p=3$ with $n=4$). 
This result was further investigated on the semilinear term including the derivatives of the solution 
by Hintz and Vasy  
\cite{Hintz-Vasy-2015-AnalysisPDE}.
We refer to 
\cite{Tsuchiya-Nakamura-2019-JCompApplMath} 
on numerical computations for the semiliear Klein-Gordon equation, 
and 
\cite{Nakamura-2015-JDE,Nakamura-2021-JDE} 
on the Cauchy problem  
for 
the semilinear Schr\"odinger equation 
and 
the semilinear Proca equation 
in the de Sitter spacetime.

On the blowing-up solution of \eqref{Eq-u-GaugeVariant} with the gauge variant semilinear term,
Yagdjian \cite[Theorem 1.1]{Yagdjian-2009-DCDS} considered the equation  
\[
\partial_t^2 u(t,\cdot)-e^{-2t}\Delta u(t,\cdot)
+
\left\{m_\ast^2-\left( \frac{n}{2}\right)^2\right\} u(t,\cdot)
-\Gamma(t)
\left(\int_{\brn}|u(t,y)|^p dy \right)^{\beta} |u(t,\cdot)|^p
=0
\]
under the normalization of $H=c=\hbar=1$,
and has shown that $\int_\brn u\, dx$ blows up in finite time for some small data 
when $0\le m_\ast\le n/2$ and the non-decreasing or non-increasing function 
$\Gamma\in C^1([0,\infty))$ satisfies 
\beq
\label{Gagdjian-Gamma}
\Gamma(t)
\ge
\begin{cases}
Ce^{-\sqrt{(n/2)^2-m_\ast^2}(p(\beta+1)-1)t}t^{2+\varepsilon} & \mbox{if}\ \ m_\ast<\frac{n}{2}, 
\\
Ct^{-1-p(\beta+1)} & \mbox{if}\ \ m_\ast=\frac{n}{2}
\end{cases}
\eeq
for $t\ge0$, 
where $p>1$, $\beta>1/p-1$, $\varepsilon>0$ and $C>0$.
We remark that when 
$\beta=0$,  
the weighted function  
$\Gamma(t)=e^{-n(p-1)t/2}$ in \eqref{Eq-u-GaugeVariant} can be taken for $m_\ast^2<0$,
namely, for the purely imaginary number $m_\ast\in i\br$ with $m\neq0$.
He has also shown the estimate of the existence time of the solution from below 
in the Sobolev space $H^s(\brn)$ under the conditions $s>n/2\ge 1$ 
in \cite[Theorem 0.1]{Yagdjian-2019-JMP} when $H=1$, 
and in \cite[Theorem 0.2]{Galstian-Yagdjian-2020-NoDEA} when $H=-1$ 
with Galstian 
(see also the references in the summary \cite{Yagdjian-9999-CurvedSpacetime}).
We refer to 
\cite{Balogh-Banda-Yagdjian-2019-CommNonlinearSciNumerSimulat} 
for the numerical simulation.

Firstly, we give the following result on the blowing-up solution in finite time 
under some conditions on data for the imaginary mass $m_\ast\in i\br$.
The result for $H>0$ and $m_\ast\neq0$ is due to Yagdjian \cite{Yagdjian-2009-DCDS} 
since his method is also applicable to the imaginary mass  
(see the introduction in \cite{Balogh-Banda-Yagdjian-2019-CommNonlinearSciNumerSimulat}).
We extend it to the case $H<0$ and $m_\ast=0$.

\begin{proposition}
\label{Thm-BlowUp}
Let $n\ge1$, $1<p<\infty$, $m_\ast\in i\br$.
Let $H>0$ or $H<0$.
Let $u_0, u_1\in L^1(\brn)$ be the functions which have compact supports and satisfy 
$w_0:=\int_{\brn} u_0(x) dx\ge0$, $w_1:=\int_{\brn} u_1(x)dx>0$, and 
\[
w_1\ge c w_0
\left\{
-\left(\frac{m_\ast c}{\hbar}\right)^2+\left(\frac{nH}{2c}\right)^2
\right\}^{1/2}.
\]
Let $w_0>0$ when $m_\ast=0$.
Then the solution of the Cauchy problem 
\beq
\label{Cauchy-BlowUp}
\begin{cases}
c^{-2}\partial_t^2u-e^{-2Ht}\Delta u
+
\left\{\left(\frac{m_\ast c}{\hbar}\right)^2-\left(\frac{nH}{2c}\right)^2\right\} u
-
e^{-n(p-1)Ht/2} |u|^p=0, \\
u(0,\cdot)=u_0(\cdot),\ \ \partial_t u(0,\cdot)=u_1(\cdot)
\end{cases}
\eeq
blows up in finite time in the space $L^1(\brn)$.
Namely, $\int_{\brn} u(t,x)dx$ blows up as $t\to T$ for some positive number $0<T<\infty$. 
\end{proposition}

\vspace{10pt}

We have shown the blowing-up solution for the gauge variant semilinear term 
with the negative sign in \eqref{Cauchy-BlowUp}.
Next, we consider the gauge invariant case with the positive sign $\lambda>0$, 
and we show the global solution for small data when $H>0$ is small.
Put $m_\ast=im$ in \eqref{Def-L-V} with $m\in \br$, $p=3$ and $\lambda>0$. 
The potential $V$ is rewritten as  
\[
V(\phi)
=
-
\frac{1}{2}
\left(\frac{mc}{\hbar}\right)^2|\phi|^2
+
\frac{\lambda}{4}|\phi|^4.
\]
This potential is known as the double well potential or the Mexican hat potential, and it has the minimum when 
\beq
\label{Def-r0}
|\phi|=r_0:=\frac{|m|c}{\sqrt{\lambda} \hbar},
\eeq
while $\phi=0$ gives the local maximum for $m\neq0$. 
It is expected that the solution around $\phi=0$ is unstable, and it is stable around $|\phi|=r_0$,
which causes the spontaneous symmetry breaking 
from $\phi=0$ to $|\phi|=r_0$.
In this paper, we characterize this breaking from the viewpoint of the Cauchy problem of 
\eqref{KG-Cauchy} which is rewritten as  
\beq
\label{EL-Eq}
\frac{1}{c^2}\partial_t^2 \phi
+\frac{nH}{c^2}\partial_t \phi
-e^{-2Ht}\Delta \phi
-\left(\frac{mc}{\hbar}\right)^2 \phi+\lambda |\phi|^2\phi=0
\eeq
(see \cite[Chapter 17]{Elbaz-1998-Springer} in the case of $H=0$).
This equation \eqref{EL-Eq} is transformed into 
\[
\frac{1}{c^2}\partial_t^2\phi+\frac{nH}{c^2}\partial_t \phi-e^{-2Ht} \Delta \phi
+
\lambda
\left\{
|\phi|^2+r_0(2\phi\real \phi+|\phi|^2)+2r_0^2\real \phi
\right\}
\]
by the shift $\phi\to \phi+r_0$ 
(see Lemma \ref{Lem-2}, below).
This equation is rewritten as 
\beq
\label{Cauchy}
c^{-2}\partial_t^2 u-e^{-2Ht} \Delta u
-\left(\frac{nH}{2c}\right)^2 u
+2\left(\frac{mc}{\hbar}\right)^2\real u
+h(u)=0
\eeq
by the change of the function $u=e^{nHt/2}\phi$ 
(see Lemma \ref{Lem-3}), 
where we have put 
\beq
\label{Def-h}
h(u):=
\lambda|u|^2u e^{-nHt} 
+
\lambda r_0(2u\real u+|u|^2) e^{-nHt/2}.
\eeq
We consider the Cauchy problem of \eqref{Cauchy} for the initial data given by 
$u(0,\cdot)=u_0(\cdot)$ and $\partial_t u(0,\cdot)=u_1(\cdot)$, 
and we show the problem is well-posed.
We observe how the Hubble constant affects on the problem.

The mechanism of the spontaneous symmetry breaking is used in the study of phase transitions 
(see  \cite{Huang-1988-JohnWiley}).
Faccioli and Salasnich 
\cite{Faccioli-Salasnich-2018-Symmetry} 
considered it for the Gross-Pitaevskii equation and also for the cubic nonlinear Klein-Gordon equation, 
and studied the spectrum of the superfluid phase of bosonic gases.
Honda and Choptuik \cite{Honda-Choptuik-2002-PhysRevD} considered the 
monotonically increasingly boosted coordinates with $n=3$ in \eqref{EL-Eq} to study localized and unstable solutions like {\it oscillon}.  
The equation \eqref{EL-Eq} reduces to the equation 
\[
\partial_t^2\phi-\partial_x^2\phi
-
\phi+
\phi^3=0,
\]
which has the potential $V(\phi)=-\phi^2/2+\phi^4/4$, 
when $H=0$ and $n=m=c=\hbar=\lambda=1$ for the real-valued function $\phi$.
This equation has the kink solution $\phi=\tanh (x/\sqrt{2})$ 
and appears in the $\phi^4$-theory,
which has been considered in the statistical mechanics, 
the condensed-matter physics, 
and the topological quantum field theory 
(see 
\cite{Barashenkov-Oxtoby-2009-PhysRevE, 
Bishop-Schnerider-1978-Springer, 
Rajarman-1982-NorthHolland}).

We denote the Lebesgue space by $L^q(I)$ for an interval $I\subset \br$ 
and $1\le q\le \infty$ with the norm 
\[
\|u\|_{L^q(I)}:=
\begin{cases}
\left\{
\int_I |u(t)|^q dt
\right\}^{1/q} & \mbox{if}\ \ 1\le q<\infty,
\\
{\mbox{ess.}\sup}_{t\in I} |u(t)| 
& \mbox{if}\ \ q=\infty.
\end{cases}
\]
We use the Sobolev space $H^\mu(\brn)$, the homogeneous Sobolev space $\dot{H}^\mu(\brn)$, 
and the homogeneous Besov space 
$\dot{B}^\mu_{r,s}(\brn)$ of order $\mu\ge0$ for $1\le r, s\le \infty$ 
(see \cite{Bergh-Lofstrom-1976-Springer} for their definitions).

%

We consider the case $H\ge0$.
For $\mu\ge0$, $0<T\le\infty$ and $Q\ge0$, 
define the function space $X^\mu$ given by 
\begin{eqnarray*}
\|u\|_{\dot{X}^\mu(T)}
&:=&
c^{-1}\|\partial_t u\|_{L^\infty((0,T),\dot{H}^\mu(\brn))}
+
\|e^{-Ht}\nabla u\|_{L^\infty((0,T),\dot{H}^\mu(\brn))}
\\
&&
+
\sqrt{Q} \|u\|_{L^\infty((0,T),\dot{H}^\mu(\brn))}
+
\sqrt{H} \|e^{-Ht}\nabla u\|_{L^2((0,T),\dot{H}^\mu(\brn))},
\\
X^\mu(T)
&:=&
\{u;\ \|u\|_{\dot{X}^\nu}<\infty\ \ \mbox{for}\ \ \nu=0,\mu_0,\mu\},
\\
\dot{D}^\mu
&:=&
c^{-1}\|u_1\|_{\dot{H}^\mu(\brn)}
+
\|\nabla u_0\|_{\dot{H}^\mu(\brn)}
+
\sqrt{Q} \|u_0\|_{\dot{H}^\mu(\brn)}.
\end{eqnarray*}
We define the metric in $X^\mu(T)$ by $d(u,v):=\|u-v\|_{\dot{X}^0(T)}$ for $u,v\in X^\mu(T)$.
We have the following results on the existence of local and global solutions, and the asymptotic behaviors of global solutions.

\begin{theorem}
\label{Thm-1}
Let $n\ge1$, $m\in \br$, $0\le H<2\sqrt{2} |m|c^2/n\hbar$.
Put 
\beq
\label{Def-Q}
Q:=2\left(\frac{mc}{\hbar}\right)^2-\left(\frac{nH}{2c}\right)^2.
\eeq
Let $\lambda>0$, and let $r_0$ be defined by \eqref{Def-r0}.
Let $\mu_0$ and $\mu$ satisfy 
\beq
\label{Assume-Mu}
\max\left\{0,\frac{n-3}{2}\right\}\le \mu_0<n/2, \ \ \mu_0\le \mu<\infty.
\eeq
Then the following results hold.

(1) 
For any real-valued functions $u_0\in H^{\mu+1}(\brn)$ and $u_1\in H^\mu(\brn)$, 
there exist $0<T<\infty$ and a unique solution 
$u\in C([0,T),H^{\mu+1}(\brn))\cap C^1([0,T),H^\mu(\brn))\cap X^\mu(T)$
of \eqref{Cauchy}. 

(2) 
The solution $u$ obtained in (1) is continuously dependent on the data.
Namely, 
$d(u,v)\to 0$ as $v_0\to u_0$ in ${H}^1(\brn)$ and 
$v_1\to u_1$ in $L^2(\brn)$ for the solution $v$ obtained in (1) 
for the data $v_0:=v(0,\cdot)$ and $v_1:=\partial_t v(0,\cdot)$.

(3) 
Assume that the following (i) or (ii) holds. 
Then the solution $u$ obtained in (1) is the global solution 
if $\dot{D}^{\mu_0}$ is sufficiently small.
Namely, $T=\infty$.
 
(i) $H>0$, $\mu_0=0$, $n\ge4$. 

(ii) $H>0$, $\mu_0>0$, $n\ge1$.

(4) 
For the global solution $u$ obtained in (3), 
put 
\begin{eqnarray}
&&
u_{+0}:=u_0+c^2\int_0^\infty K_1(s) h(u)(s) ds,
\nonumber\\
&&
u_{+1}:=u_1-c^2\int_0^\infty K_0(s) h(u)(s) ds,
\label{Def-Asymptotic}\\
&&u_+:=K_0(t) u_{+0}+K_1(t) u_{+1},
\nonumber
\end{eqnarray}
where $K_0$ and $K_1$ are the operators defined by \eqref{Def-K}, below.
Then $u_{+0}\in H^\mu(\brn)$, $u_{+1}\in H^{\mu-1}(\brn)$, and 
$u$ has the asymptotic behaviors given by 
\[
\|u(t)-u_+(t)\|_{H^{\mu-1}(\brn)}\to 0,\ \ 
\|\partial_t \left(u(t)-u_+(t)\right)\|_{H^{\mu-1}(\brn)}\to 0.
\]

(5) 
The solution $u$ obtained in (1) is the global solution 
when $H=0$ and $\mu_0=0$.
Namely, $T=\infty$.
\end{theorem}

\vspace{10pt}

Next, we consider the case $H<0$. 
For $\mu\ge0$, $T>0$ and $Q\ge0$, 
define the function space $X^\mu$ given by 
\begin{eqnarray}
\|u\|_{\dot{X}^\mu(T)}
&:=&
c^{-1}\|e^{Ht}\partial_t u\|_{L^\infty((0,T),\dot{H}^\mu(\brn))}
+
\|\nabla u\|_{L^\infty((0,T),\dot{H}^\mu(\brn))}
\nonumber\\
&&
+
\sqrt{Q} \|e^{Ht} u\|_{L^\infty((0,T),\dot{H}^\mu(\brn))}
+
c^{-1}\sqrt{-H} \|e^{Ht}\partial_t u\|_{L^2((0,T),\dot{H}^\mu(\brn))}
\nonumber
\\
&&
+
\sqrt{-HQ} \|e^{Ht} u\|_{L^2((0,T),\dot{H}^\mu(\brn))},
\label{Thm-2-X}
\\
X^\mu(T)
&:=&
\{u;\ \|u\|_{\dot{X}^\nu}<\infty\ \ \mbox{for}\ \ \nu=0,\mu_0,\mu\},
\nonumber
\\
\dot{D}^\mu
&:=&
c^{-1}\|u_1\|_{\dot{H}^\mu(\brn)}
+
\|\nabla u_0\|_{\dot{H}^\mu(\brn)}
+
\sqrt{Q} \|u_0\|_{\dot{H}^\mu(\brn)}.
\nonumber
\end{eqnarray}
We define the metric in $X^\mu(T)$ by $d(u,v):=\|u-v\|_{\dot{X}^0(T)}$ for $u,v\in X^\mu(T)$.

\begin{theorem}
\label{Thm-2}
Let $n\ge1$, $-2\sqrt{2} |m|c^2/n\hbar<H<0$.
Let $Q$ be defined by \eqref{Def-Q}.
Let $\lambda>0$, and let $r_0$ be defined by \eqref{Def-r0}.
Let $\mu_0$ and $\mu$ satisfy \eqref{Assume-Mu}.

(1) 
For any real-valued functions $u_0\in H^{\mu+1}(\brn)$ and $u_1\in H^\mu(\brn)$, 
there exist $0<T<\infty$ and a unique solution 
$u\in C([0,T),H^{\mu+1}(\brn))\cap C^1([0,T),H^\mu(\brn))\cap X^\mu(T)$
of \eqref{Cauchy}. 
Here, $T>0$ can be taken as the number which satisfies 
\begin{eqnarray}
&&
C\lambda c(-H)^{-1}
\left\{
\left(
\frac{e^{-4(1+\mu_0)HT}-1}{4(1+\mu_0)}
\right)^{1/2}
Q^{(n-3-2\mu_0)/2} C_0\dot{D}^{\mu_0}
\right.
\nonumber
\\
&&
\ \ \ \ 
+
\left.
r_0
\left(
\frac{e^{-2(1+\mu_0)HT}-1}{2(1+\mu_0)}
\right)^{1/2}
Q^{(n-4-2\mu_0)/4} 
\right\}
\le \frac{1}{2}
\label{Thm-2-1000}
\end{eqnarray}
for some universal constant $C_0>0$ and $C>0$.

(2) 
The solution $u$ obtained in (1) is continuously dependent on the data.
Namely, 
$d(u,v)\to 0$ as $v_0\to u_0$ in $H^1(\brn)$ and 
$v_1\to u_1$ in $L^2(\brn)$ for the solution $v$ obtained in (1) 
for the data $v_0:=v(0,\cdot)$ and $v_1:=\partial_t v(0,\cdot)$.
\end{theorem}

On Theorems \ref{Thm-1} and \ref{Thm-2}, 
we remark that the estimate of the lifespan of time-local solutions from below has been shown 
in 
\cite[(iii) in Theorem 0.1]{Yagdjian-2019-JMP} when $H>0$, 
and 
in 
\cite[Theorem 0.2]{Galstian-Yagdjian-2020-NoDEA} when $H<0$,  
for data with high regularity such as $\mu>n/2\ge1$.
Theorem \ref{Thm-1} shows the existence of global solutions for small rough data, 
and that the asymptotic behaviors are given by the free solutions defined by 
\eqref{Def-Asymptotic} when $H=0$ or $H>0$ is small, 
while Theorem \ref{Thm-2} gives more explicit estimate of the lifespan of time-local solutions from below for rough data.

We denote the inequality $A\le CB$ for some constant $C>0$ which is not essential for the argument by $A\lesssim B$.
This paper is organized as follows.
In Section \ref{Section-Pre}, 
we collect several results on the derivation of the first equation in \eqref{Cauchy} as the Euler-Lagrange equation from a Lagrangian in the de Sitter spacetime, 
the energy estimate of the equation, 
which are required to prove Proposition \ref{Thm-BlowUp}, Theorems \ref{Thm-1} and \ref{Thm-2} 
in Sections 
\ref{Section-Proof-Thm-BlowUp}, \ref{Section-Proof-Thm-1} and \ref{Section-Proof-Thm-2}, 
respectively.

%

\newsection{Preliminaries}
\label{Section-Pre}
In this section, we collect several results to prove the results in the previous section.
We introduce the following fundamental result with its proof 
since Lemme \ref{Lem-2} is based on it.

\begin{lemma}
\label{Lem-1}
Let $m\in \br$.
Let $\lambda\in \br$.
Let $H\in \br$, and put $(g_{\alpha\beta}):=\diag(-c^2,e^{2Ht},\cdots,e^{2Ht})$.
Let $L$ be the Lagrangian density defined by 
\[
L(\phi):=-\frac{1}{2} \overline{\partial^\alpha} \phi \partial_\alpha \phi
+\frac{1}{2} \left(\frac{mc}{\hbar}\right)^2 |\phi|^2-\frac{\lambda}{4}|\phi|^4.
\]
Let $g$ denote the determinant of the matrix $(g_{\alpha\beta})$.
Then the Euler-Lagrange equation of the action 
\[
\int_\br \int_{\brn} L(\phi) \sqrt{-g} dx dt
\] 
is given by \eqref{EL-Eq}.
\end{lemma}

\begin{proof}
Since we have 
\[
\delta(\overline{\partial^\alpha \phi}\partial_\alpha \phi)
=
2\partial_\alpha(\real(\overline{\delta \phi}\partial^\alpha \phi))
-
2\real(\overline{\delta \phi}\partial_\alpha\partial^\alpha\phi),
\]
\[ 
\delta|\phi|^2=2\real(\overline{\delta\phi}\phi),
\ \ 
\delta|\phi|^4=4|\phi|^2\real(\overline{\delta\phi}\phi),
\]
we obtain 
\[
\delta L(\phi)
=
-\partial_\alpha(\real(\overline{\delta \phi}\partial^\alpha \phi))
+
\real(\overline{\delta\phi} E),
\]
where we have put 
\[
E:=\partial_\alpha\partial^\alpha \phi
+
\left(\frac{mc}{\hbar}\right)^2 \phi-\lambda |\phi|^2\phi.
\]
So that, we have 
\begin{eqnarray*}
\delta\int_\br \int_{\brn} L(\phi) \sqrt{-g} dx dt
&=&
-
\int_\br \int_{\brn} 
\partial_\alpha(\real(\overline{\delta \phi}\partial^\alpha \phi)\sqrt{-g}) dx dt
\\
&&
+
\int_\br \int_{\brn} 
\real(\overline{\delta\phi}F)
\sqrt{-g} dx dt,
\end{eqnarray*}
where $F$ is defined by 
\[
F:=E+\partial^\alpha\phi\frac{\partial_\alpha \sqrt{-g}}{\sqrt{-g}}=E-\frac{nH}{c^2}\partial_0\phi,
\]
which yields the required equation $F=0$ as the Euler-Lagrange equation.
\end{proof}

\begin{lemma}
\label{Lem-2}
Let $\lambda>0$.
Under the assumptions in Lemma \ref{Lem-1}, put 
\[
V(r):=-\frac{1}{2}\left(\frac{mc}{\hbar}\right)^2 r^2+\frac{\lambda}{4} r^4.
\]
Let $r_0$ be the number defined by \eqref{Def-r0}.
Then the Euler-Lagrange equation of the action 
$\int_\br \int_{\br^{n}} L(\phi+r_0) \sqrt{-g} dxdt$ is given by 
\beq
\label{Lem-2-1000}
\frac{1}{c^2}\partial_t^2\phi
-e^{-2Ht}\Delta\phi
+\frac{nH}{c^2} \partial_t \phi
+J=0,
\eeq
where  
\beq
\label{Lem-2-2000}
J:=\lambda
\left\{
|\phi|^2\phi+r_0(2\phi\real \phi+|\phi|^2)
+2r_0^2\real \phi 
\right\}.
\eeq
\end{lemma}

\begin{proof}
We note that $r_0$ gives the minimum of $V$ by $V'(r)=\lambda r(r+r_0)(r-r_0)$.
We have 
\[
V(|\phi+r_0|)
=
\frac{\lambda}{4}
\left(
|\phi|^4+4r_0 |\phi|^2\real \phi+4r_0^2(\real \phi)^2-r_0^4
\right)
\]
and 
$\delta V(|\phi+r_0|)=\real(J\overline{\delta\phi})$ 
and 
$\delta L(\phi+r_0)=-\real(
\overline{\partial_\alpha \delta\phi}\partial^\alpha \phi
)
-
\real(J\overline{\delta\phi})$  
by direct calculations, 
where we have used 
\begin{eqnarray*}
&&
\delta|\phi|^2=2\real(\phi\overline{\delta\phi}),
\ \ 
\delta|\phi|^4=4\real(|\phi|^2 \phi\overline{\delta\phi}),
\ \ 
\delta(\real \phi)^2=2\real(\real \phi\overline{\delta\phi}),
\\
&&
\delta(|\phi|^2\real \phi)
=
\real\left(\left(2\phi\real \phi+|\phi|^2\right)\overline{\delta\phi}\right),
\ \ 
\delta(\overline{\partial^\alpha \phi}\partial_\alpha \phi)
=2\real(\overline{\partial_\alpha \delta\phi} \partial^\alpha\phi).
\end{eqnarray*}
Thus, we obtain 
\begin{multline*}
\delta \int_\br \int_{\brn} L(\phi+r_0) \sqrt{-g} dxdt
=
-
\real\int_\br\int_{\br^{n}} 
\partial_\alpha (\overline{\delta\phi}\partial^\alpha \phi \sqrt{-g}) dx dt
\\
+
\real
\int_\br\int_{\br^{n}} 
K\overline{\delta\phi}\sqrt{-g}dx dt,
\end{multline*}
where we have put 
$K:=\partial_\alpha(\partial^\alpha \phi \sqrt{-g})/\sqrt{-g}-J$.
So that, the Euler-Lagrange equation is given by $K=0$, 
which yields the required equation \eqref{Lem-2-1000}.
\end{proof}

\begin{lemma}
\label{Lem-3}
Under the assumptions in Lemma \ref{Lem-2}, 
put $u:=e^{nHt/2} \phi$.
Then the equation \eqref{Lem-2-1000} is rewritten as \eqref{Cauchy}, 
where $h(u)$ is defined by \eqref{Def-h}.
\end{lemma}

\begin{proof}
Let $K$ be defined by the left hand side in \eqref{Lem-2-1000}.
Put $\eta:=-nH/2$.
Since we have 
\[
Je^{-\eta t}=
\lambda|u|^2ue^{2\eta t}+\lambda r_0(2u\real u+|u|^2)e^{\eta t}
+2\lambda r_0^2 \real u,
\]
we obtain 
\[
-Ke^{-\eta t}=
\frac{1}{c^2}\partial_t^2 u-e^{-2Ht} \Delta u
-
\left(\frac{nH}{2c}\right)^2 u
+h(u)+2\left(\frac{mc}{\hbar}\right)^2 \real u,
\]
which shows $K=0$ gives the required equation \eqref{Cauchy}.
\end{proof}

\begin{lemma}
\label{Lem-4}
Let $H\ge 0$, $\lambda\in \br$ and $r_0\in \br$.
Consider the Cauchy problem 
\beq
\label{Lem-4-1000}
\begin{cases}
\frac{1}{c^2}\partial_t^2 u-e^{-2Ht} \Delta u
-
\left(\frac{nH}{2c}\right)^2 u
+2\lambda r_0^2 \real u+h
=0,
\\
u(0,\cdot)=u_0(\cdot),
\ \ 
\partial_t u(0,\cdot)=u_1(\cdot)
\end{cases}
\eeq
for any given function $h$ which decay rapidly at spatial infinity.
Then the following results hold.

(1) 
$\partial_\alpha e^\alpha+e^{n+1}+e^{n+2}=0$, 
where 
\begin{eqnarray*}
&& 
e^0:=\frac{1}{2c^2}|\partial_t u|^2-\frac{1}{2}
\left(\frac{nH}{2c}\right)^2|u|^2+\lambda r_0^2(\real u)^2+\frac{e^{-2Ht}}{2}|\nabla u|^2,
\\
&& 
e^j:=-e^{-2Ht}\real \left(\overline{\partial_t u}\partial_j u\right),
\\
&& 
e^{n+1}:=He^{-2Ht}|\nabla u|^2, \ \ e^{n+2}:=\real \left(\overline{\partial_t u}h\right).
\end{eqnarray*}

(2) 
If $u$ is real-valued and 
$Q:=2\lambda r_0^2-(nH/2c)^2\ge0$, 
then the following estimate hold; 
\begin{eqnarray*}
&&
\frac{1}{c}\|\partial_t u\|_{L^\infty((0,\infty),L^2(\brn))}
+
\sqrt{Q} \|u\|_{L^\infty((0,\infty),L^2(\brn))}
+
\|e^{-Ht}\nabla u\|_{L^\infty((0,\infty),L^2(\brn))}
\\
&&+
\sqrt{H}\|e^{-Ht}\nabla u\|_{L^2((0,\infty)\times\brn)}
\\
&\lesssim&
\frac{1}{c}\|u_1\|_{L^2(\brn)}
+
\sqrt{Q} \|u_0\|_{L^2(\brn)}
+
\|\nabla u_0\|_{L^2(\brn)}
+
c\|h\|_{L^1((0,\infty),L^2(\brn))}.
\end{eqnarray*}

(3) 
If $h=h(u)$ is given by \eqref{Def-h}, then 
$\partial_0 \tilde{e}^0+\partial_j e^j+\tilde{e}^{n+1}=0$ holds,
where 
\begin{eqnarray*}
&& 
\tilde{e}^0:=e^0+\frac{\lambda}{4}|u|^4e^{-nHt}+\lambda r_0|u|^2\real u\, e^{-nHt/2}
\\
&& 
\tilde{e}^{n+1}:=e^{n+1}+\frac{\lambda nH}{4}|u|^4e^{-nHt}
+
\frac{\lambda r_0nH}{2}|u|^2\real u\, e^{-nHt/2},
\end{eqnarray*}
and $e^\alpha$ is defined in (1) for $0\le \alpha\le n+1$.
\end{lemma}

\begin{proof}
(1) 
Multiplying $\overline{\partial_t u}$ to the first equation in \eqref{Lem-4-1000} and taking its real part, we obtain the required equation by 
\begin{eqnarray*}
&&
\real\left(\overline{\partial_t u}\partial_t^2 u\right)
=\partial_t
\left(\frac{1}{2} |\partial_t u|^2\right),
\\
&&
e^{-2Ht} \real \left(\overline{\partial_t u}\Delta u\right)
=
\nabla
\left\{e^{-2Ht}\real\left(\overline{\partial_t u}\nabla u\right)\right\}
-\partial_t\left(\frac{e^{-2Ht}}{2} |\nabla u|^2\right)
-He^{-2Ht} |\nabla u|^2,
\\
&&
\real \left(\overline{\partial_t u} u\right)
=
\partial_t \left(\frac{|u|^2}{2}\right),
\ \ \real\left(\overline{\partial_t u}\real u\right)
=
\partial_t \left(\frac{|\real u|^2}{2}\right).
\end{eqnarray*}

(2) 
Integrating the both sides in $\partial_\alpha e^\alpha+e^{n+1}+e^{n+2}=0$ in (1), 
we have 
\[
\int_\brn e^0(t) dx
+
H\|e^{-Hs}\nabla u\|_{L^2((0,t),L^2)}^2
=
\int_\brn e^0(0) dx
-
\int_0^t \int_\brn e^{n+2} dxdt
\]
for $t>0$,
where we have used 
\[
\int_\brn e^0(t) dx=\frac{1}{2c^2}\|\partial_t u(t)\|_2^2+\frac{Q}{2} \|u(t)\|_2^2
+\frac{e^{-2Ht}}{2} \|\nabla u(t)\|_2^2
\]
since $u$ is real-valued.
By the H\"older inequality 
$\int_0^t\int_\brn |e^{n+2}| dx dt\le \|\partial_t u\|_{L^\infty L^2} \|h\|_{L^1L2}$, 
we obtain the required result.

(3) 
Put 
$e^0_\ast:=\tilde{e}^0-e^0$,
$e^{n+1}_\ast:=\tilde{e}^{n+1}-e^{n+1}$.
We have $\real(\overline{\partial_t u} h)=\partial_t e^0_\ast+e^{n+1}_\ast$  
by 
\begin{eqnarray*}
&&
\real\left(\overline{\partial_t u} |u|^2 u\right) e^{-nHt}
=
\partial_t
\left(\frac{1}{4} |u|^4 e^{-nHt}\right)+\frac{nH}{4} |u|^4 e^{-nHt},
\\
&&
\real
\left\{
\overline{\partial_t u} 
\left(2u\real u+|u|^2 \right) 
\right\} e^{-nHt/2}
=
\partial_t
\left( |u|^2 \real u\, e^{-nHt/2}\right)
+
\frac{nH}{2} |u|^2 \real u\, e^{-nHt/2}.
\end{eqnarray*}
So that, we obtain 
$
0=\partial_\alpha e^\alpha+e^{n+1}+e^{n+2}
=\partial_0 \tilde{e}^0+\partial_j e^j+\tilde{e}^{n+1} 
$ 
by (1), 
which is the required result.
\end{proof}

\begin{lemma}
\label{Lem-5}
Let $H<0$, $\lambda\in \br$ and $r_0\in \br$.
Consider the problem \eqref{Lem-4-1000}.
Then the following results hold.

(1) 
$\partial_\alpha e^\alpha+e^{n+1}+e^{n+2}=0$, 
where 
\begin{eqnarray*}
&& 
e^0:=
\frac{e^{2Ht} }{2c^2}|\partial_t u|^2
+
\frac{1}{2}|\nabla u|^2
-
\frac{e^{2Ht} }{2}
\left(\frac{nH}{2c}\right)^2|u|^2
+
\lambda r_0^2e^{2Ht} |\real u|^2,
\\
&& 
e^j:=-\real \left(\overline{\partial_t u}\partial_j u\right),
\\
&& 
e^{n+1}
:=
-
\frac{He^{2Ht}}{c^2}|\partial_t u|^2
+
\left(\frac{nH}{2c}\right)^2 H e^{2Ht} |u|^2
-
2\lambda r_0^2 H e^{2Ht}|\real u|^2, 
\\
&& 
e^{n+2}
:=
e^{2Ht} \real \left(\overline{\partial_t u} h\right).
\end{eqnarray*}

(2) 
Let $q$ be any number with $2\le q\le \infty$, 
and let $q'$ be the conjugate number with $1/q+1/q'=1$.
If $u$ is real-valued and 
$Q:=2\lambda r_0^2-(nH/2c)^2\ge0$, 
then the following estimate hold; 
\begin{eqnarray*}
&&
\frac{1}{c}\|e^{Ht}\partial_t u\|_{L^\infty((0,\infty),L^2(\brn))}
+
\sqrt{Q} \|e^{Ht} u\|_{L^\infty((0,\infty),L^2(\brn))}
+
\|\nabla u\|_{L^\infty((0,\infty),L^2(\brn))}
\\
&&
+
\frac{\sqrt{-H}}{c} 
\|e^{Ht}\partial_t u\|_{L^2((0,\infty)\times\brn)}
+
\sqrt{-HQ} 
\|e^{Ht} u\|_{L^2((0,\infty)\times\brn)}
\\
&\lesssim&
\frac{1}{c}\|u_1\|_{L^2(\brn)}
+
\sqrt{Q} \|u_0\|_{L^2(\brn)}
+
\|\nabla u_0\|_{L^2(\brn)}
+
\frac{c}{(-H)^{1/q}} 
\|e^{Ht} h\|_{L^{q'}((0,\infty),L^2(\brn))}.
\end{eqnarray*}

(3) 
If $h=h(u)$ is given by \eqref{Def-h}, then 
$\partial_0 \tilde{e}^0+\partial_j e^j+\tilde{e}^{n+1}=0$ holds,
where 
\begin{eqnarray*}
&& 
\tilde{e}^0:=e^0+\frac{\lambda}{4}|u|^4e^{-(n-2)Ht}+\lambda r_0|u|^2\real u\, e^{-(n-4)Ht/2}
\\
&& 
\tilde{e}^{n+1}:=e^{n+1}+\frac{\lambda (n-2)H}{4}|u|^4e^{-(n-2)Ht}
+
\frac{\lambda r_0(n-4)H}{2}|u|^2\real u\, e^{-(n-4)Ht/2},
\end{eqnarray*}
and $e^\alpha$ is defined in (1) for $0\le \alpha\le n+1$.
\end{lemma}

\begin{proof}
(1) 
Multiplying $e^{2Ht}\overline{\partial_t u}$ to the first equation in \eqref{Lem-4-1000} and taking its real part, we obtain the required equation by 
\begin{eqnarray*}
&&
e^{2Ht}\real\left(\overline{\partial_t u}\partial_t^2 u\right)
=\partial_t
\left(\frac{e^{2Ht} }{2} |\partial_t u|^2\right)
-He^{2Ht}|\partial_t u|^2,
\\
&&
\real \left(\overline{\partial_t u}\Delta u\right)
=
\nabla
\real\left(\overline{\partial_t u}\nabla u\right)
-
\partial_t\left(\frac{1}{2} |\nabla u|^2\right),
\\
&&
e^{2Ht}\real \left(\overline{\partial_t u} u\right)
=
\partial_t \left(\frac{e^{2Ht}}{2}|u|^2\right)
-
He^{2Ht} |u|^2,
\\ 
&&
e^{2Ht}\real\left(\overline{\partial_t u}\real u\right)
=
\partial_t \left(\frac{e^{2Ht} |\real u|^2}{2}\right)
-
He^{2Ht}|\real u|^2.
\end{eqnarray*}

(2) 
Integrating the both sides in $\partial_\alpha e^\alpha+e^{n+1}+e^{n+2}=0$ in (1), 
we have 
\[
\int_\brn e^0(t) dx
+
\int_0^t \int_\brn e^{n+1} dx ds
=
\int_\brn e^0(0) dx
-
\int_0^t \int_\brn e^{n+2} dxdt
\]
for $t>0$, 
where we note  
\[
\int_\brn e^0(t) dx=\frac{1}{2c^2}\|e^{Ht}\partial_t u(t)\|_2^2
+
\frac{1}{2}\|\nabla u(t)\|_2^2
+
\frac{Q}{2} \|e^{Ht} u(t)\|_2^2,
\]
and
\[
\int_0^t \int_\brn e^{n+1} dx ds
=
-\frac{H}{c^2}\|e^{Hs} \partial_t u\|_{L^2L^2}^2
-HQ\|e^{Hs} u\|_{L^2L^2}^2
\]
since $u$ is real-valued.
We estimate the last term by the H\"older inequality  
\begin{eqnarray*}
\int_0^t\int_\brn |e^{n+2}| dx dt
&\le& 
\|e^{Hs}\partial_t u\|_{L^q L^2} \|e^{Hs}h\|_{L^{q'} L^2}
\\
&\le& 
\frac{ \varepsilon^2(-H)^{2/q} }{4c^2}
\|e^{Hs}\partial_t u\|_{L^qL^2}^2
+
\frac{ c^2}{\varepsilon^2(-H)^{2/q} }
\|e^{Hs}h\|_{L^{q'}L^2}^2
\end{eqnarray*} 
for any number $\varepsilon>0$.
So that, the required inequality follows from the interpolation inequality 
$\|e^{Hs}\partial_t u\|_{L^qL^2}
\le
\|e^{Hs}\partial_t u\|_{L^\infty L^2}^{1-2/q}
\|e^{Hs}\partial_t u\|_{L^2 L^2}^{2/q}$ 
with $\varepsilon>0$ taken sufficiently small.

(3) 
Put 
$e^0_\ast:=\tilde{e}^0-e^0$,
$e^{n+1}_\ast:=\tilde{e}^{n+1}-e^{n+1}$.
We have 
$e^{2Ht}\real(\overline{\partial_t u} h)=\partial_t e^0_\ast+e^{n+1}_\ast$  
by 
\[
\real\left(\overline{\partial_t u} |u|^2 u\right) e^{-(n-2)Ht}
=
\partial_t
\left(\frac{1}{4} |u|^4 e^{-(n-2)Ht}\right)+\frac{(n-2)H}{4} |u|^4 e^{-(n-2)Ht},
\]
and 
\begin{multline*}
\real
\left\{
\overline{\partial_t u} 
\left(2u\real u+|u|^2 \right) 
\right\} e^{-(n-4)Ht/2}
=
\partial_t
\left( |u|^2 \real u\, e^{-(n-4)Ht/2}\right)
\\
+
\frac{(n-4)H}{2} |u|^2 \real u \, e^{-(n-4)Ht/2}.
\end{multline*}
So that, we obtain 
$
0=\partial_\alpha e^\alpha+e^{n+1}+e^{n+2}
=\partial_0 \tilde{e}^0+\partial_j e^j+\tilde{e}^{n+1} 
$ 
from (1) as required.
\end{proof}

We confirm that the Euler-Lagrange equation \eqref{Lem-2-1000} with \eqref{Lem-2-2000} is obtained from 
\eqref{EL-Eq} by the shift of the function $\phi$ as follows.

\begin{lemma}
\label{Lem-6}
For $\lambda>0$ and $r_0$ defined by \eqref{Def-r0},
the equation \eqref{Lem-2-1000} with \eqref{Lem-2-2000} 
is obtained from \eqref{EL-Eq} 
with $\phi$ replaced by $\phi+r_0$.
\end{lemma}

\begin{proof}
The result follows from a direct calculation by 
\[
\lambda|\phi+r_0|^2(\phi+r_0)
=
\lambda
\left\{|\phi|^2\phi+r_0(2\phi\real\phi+|\phi|^2)+2r_0^2\real\phi\right\}
+\lambda r_0^2(\phi+r_0)
\]
and $\lambda r_0^2=(mc/\hbar)^2$.
\end{proof}

To express the solution of the differential equation as the integral equation, 
we recall the following fundamental results for ordinary differential equations 
(see, e.g., \cite{Nakamura-2019-ODE}).
Put $D_t:=d/dt$.

\begin{lemma}
\label{Lem-7}
For any fixed nonnegative function 
$\widetilde{a}\in C([0,T))$ for $T>0$, 
let $\rho_0$ and $\rho_1$ be the solutions of the Cauchy problem 
\beq
\label{Lem-7-1000}
\left\{
\begin{array}{l}
\left(D_t^2+\widetilde{a}(t)\right) \rho_j(t)=0 \ \  \mbox{for} \ \ t\in [0,T), \\
\rho_j(0)=\delta_{0j}, \ \ D_t\rho_j(0)=\delta_{1j}
\end{array}
\right.
\eeq
for $j=0,1$, 
where 
$\delta_{00}=\delta_{11}=1$ and $\delta_{01}=\delta_{10}=0$.
Let $b\in L^1((0,T))$, and let $\rho$ be the solution of the equation  
\beq
\label{Lem-7-2000}
(D_t^2+\widetilde{a}(t))\rho(t)=b(t)
\eeq
for $t\in [0,T)$.
Put $\Phi=
\begin{pmatrix}
\rho_0 & \rho_1 \\
D_t \rho_0 & D_t\rho_1
\end{pmatrix}$.
Then the following results hold.

(1) 
$\det \Phi=1$.

(2) 
$\rho$ is given by 
\[
\begin{pmatrix}
\rho(t) \\
D_t\rho(t)
\end{pmatrix}
=\Phi(t)
\begin{pmatrix}
\rho(0) \\
D_t\rho(0)
\end{pmatrix}
+
\int_0^t
\Phi(t)\Phi(s)^{-1} 
\begin{pmatrix}
0 \\
b(s)
\end{pmatrix}
ds,
\]
which is rewritten as 
\begin{eqnarray}
\rho(t)&=&\rho_0(t)\rho(0)+\rho_1(t)D_t\rho(0)+\int_0^t \rho_{12}(t,s) b(s) ds,
\label{Lem-7-3000}
\\
D_t\rho(t)&=&D_t\rho_0(t)\rho(0)+D_t\rho_1(t)D_t\rho(0)+\int_0^t \rho_{22}(t,s) b(s) ds,
\label{Lem-7-4000}
\end{eqnarray}
where $\rho_{12}$ and $\rho_{22}$ are defined by 
\begin{eqnarray}
\rho_{12}(t,s)&:=&-\rho_0(t)\rho_1(s)+\rho_1(t)\rho_0(s),
\label{Lem-7-4100}
\\ 
\rho_{22}(t,s)&:=&-D_t\rho_0(t)\rho_1(s)+D_t\rho_1(t)\rho_0(s).
\label{Lem-7-4200}
\end{eqnarray}

(3) 
If $D_t \widetilde{a}\le 0$, then 
\beq
\label{Lem-7-4}
|\rho_0(\cdot)|\le 
\sqrt{ \frac{\widetilde{a}(0)}{\widetilde{a}(\cdot)} },
\ \ 
|D_t\rho_0(\cdot)|\le \sqrt{ \widetilde{a}(0) },
\ \ 
|\rho_1(\cdot)|\le 
\frac{1}{ \sqrt{\widetilde{a}(\cdot)} },
\ \ 
|D_t\rho_1(\cdot)|\le 1.
\eeq

(4) 
If $D_t \widetilde{a}\ge 0$, then 
\beq
\label{Lem-7-5}
|\rho_0(\cdot)|\le 1, 
\ \ 
|D_t\rho_0(\cdot)|\le \sqrt{ \widetilde{a}(\cdot) },
\ \ 
|\rho_1(\cdot)|\le 
\frac{1}{ \sqrt{\widetilde{a}(0)} },
\ \ 
|D_t\rho_1(\cdot)|\le 
\sqrt{ \frac{\widetilde{a}(\cdot)}{\widetilde{a}(0)} }.
\eeq

(5) 
$\rho\in C([0,T))$.
If $\widetilde{a}\in C([0,T))$ and $b\in C([0,T))$, then $\rho\in C^2([0,T))$.
\end{lemma}

\vspace{10pt}

For $H\in \br$ and $Q\in \br$, 
let $\rho_0=\rho_0(t,\xi)$ and $\rho_1=\rho_1(t,\xi)$ be the functions 
obtained by Lemma \ref{Lem-7} 
putting 
\beq
\label{Def-aTilde}
\widetilde{a}=\widetilde{a}(t,\xi)
:=
\frac{c^2}{e^{2Ht}}\cdot 
\sum_{j=1}^n (\xi^j)^2+c^2Q
\eeq
for $\xi=(\xi^1, \cdots,\xi^n)$.
Put 
\beq
\label{Def-K}
\begin{array}{l}
K_0(t):=F^{-1}\rho_0(t,\cdot)F,
\ \  
K_1(t):=F^{-1}\rho_1(t,\cdot)F,
\\ 
K(t,s):=c^2\left\{-K_0(t)K_1(s)+K_1(t)K_0(s)\right\}
\end{array}
\eeq
for $t,s\in \br$,
where $F$ and $F^{-1}$ denote the Fourier transform and its inverse for $(x^1,\cdots,x^n)$.
Then the Cauchy problem 
\beq
\label{Lem-8-1000}
\frac{1}{c^2}\partial_t^2 u-e^{-2Ht} \Delta u +Qu+h=0, 
\ \ 
u(0)=u_0,
\ \ 
\partial_t u(0)=u_1
\eeq
for a given function $h$ on $\br^{1+n}$ can be regarded as the solution of the integral equation 
\beq
\label{Cauchy-Int}
u(t)=K_0(t)u_0+K_1(t)u_1-\int_0^t K(t,s) h(s) ds.
\eeq

By the estimates \eqref{Lem-7-4}, \eqref{Lem-7-5} and the Plancherel theorem for the Fourier transform, 
we obtain the following results 
(see e.g., \cite[Lemma 4.4]{Nakamura-2021-JDE}).

\begin{lemma}
\label{Lem-10}
Let $\mu\in \br$, 
and let $h\in L^1((0,\infty),H^\mu(\brn))$.
Let $K_0$ and $K_1$ be the operators defined by \eqref{Def-K}.
Put 
\begin{eqnarray*}
u_{+0}&:=&u_0+c^2\int_0^\infty K_1(s) h(s) ds,
\\
u_{+1}&:=&u_1-c^2\int_0^\infty K_0(s) h(s) ds,
\\
u_+(t)&:=& K_0(t)u_{+0}+K_1(t)u_{+1}
\end{eqnarray*}
for $t\ge0$.
Let $u$ be the solution of \eqref{Cauchy-Int}.
Then the following estimates hold.

\[
(1)
\ \  
\|u_{+0}\|_{H^\mu(\brn)}
\lesssim 
\|u_0\|_{H^\mu(\brn)}+c\int_0^\infty\|h(s)\|_{H^\mu(\brn)} ds
\]

\[
(2) 
\ \ 
\|u_{+1}\|_{H^{\mu-1}(\brn)}
\lesssim 
\|u_1\|_{H^{\mu-1}(\brn)}+c^2\int_0^\infty\|h(s)\|_{H^\mu(\brn)} ds
\]

\[
(3) 
\ \ 
\|u_{+}\|_{H^{\mu-1}(\brn)}
\lesssim 
\|u_0\|_{H^\mu(\brn)}
+\frac{1}{c}\|u_1\|_{H^{\mu-1}(\brn)}
+c\int_0^\infty\|h(s)\|_{H^\mu(\brn)} ds
\]

\[
(4) 
\ \ 
\|\partial_t u_{+}\|_{H^{\mu-1}(\brn)}
\lesssim 
c\|u_0\|_{H^\mu(\brn)}
+\|u_1\|_{H^{\mu-1}(\brn)}
+c^2\int_0^\infty\|h(s)\|_{H^\mu(\brn)} ds
\]

\[
(5) 
\ \ 
\|u(t)-u_{+}(t)\|_{H^{\mu-1}(\brn)}
\lesssim 
\frac{1}{c} \int_t^\infty\|h(s)\|_{H^\mu(\brn)} ds
\]

\[
(6) 
\ \ 
\|\partial_t\left(u(t)-u_{+}(t)\right)\|_{H^{\mu-1}(\brn)}
\lesssim 
c^2 \int_t^\infty\|h(s)\|_{H^\mu(\brn)} ds
\]
\end{lemma}

%

\newsection{Proof of Proposition \ref{Thm-BlowUp} }
\label{Section-Proof-Thm-BlowUp}
Put 
\[
Q:=\left(\frac{m_\ast c}{\hbar}\right)^2-\left(\frac{nH}{2c}\right)^2
\ \ \mbox{and}\ \ 
M:=\sqrt{-Q}
\]
which satisfies $M\ge0$ by $Q\le 0$.
Put $w(t):=\int_{\brn} u(t,x)dx$ for $t\ge0$. 
It suffices to show that $w$ blows up in finite time.
Integrating the first equation in \eqref{Cauchy-BlowUp}, we have 
\beq
\label{Thm-1-Proof-1000}
\frac{1}{c^{2}} \partial_t^2w +Qw-h=0,
\eeq
where $h:=e^{-n(p-1)Ht/2}\int_\brn |u|^pdx$ and 
we have used the divergence theorem.
So that, $w$ is written as 
\beq
\label{Thm-1-Proof-1500}
w(t)=(\cosh cM t) w_0
+
\frac{\sinh cM t}{cM} w_1
+
c^2\int_0^t \frac{\sinh cM (t-s)}{cM} h(s)ds, 
\eeq
and $w$ satisfies $w(t)\ge0$ and $\partial_t w(t)\ge0$ for $t\ge0$ 
by $w_0\ge0$, $w_1\ge0$, $h\ge0$.
By the finite speed of the propagation, we may assume 
that the support of $u(t,\cdot)$ is in the ball of the radius $r(t):=r_0+c(1-e^{-Ht})/H$ 
for some number $r_0>0$ with $\supp u_0\cup \supp u_1\subset\{x\in \brn;\, |x|\le r_0\}$.
By this support condition and the H\"older inequality, we have 
\[
w(t)^p\le \{\omega_n r(t)^n\}^{p-1} \int_\brn |u(t,x)|^p dx,
\]
which yields 
\beq
\label{Thm-1-Proof-2000}
h(t)\ge b(t) |w(t)|^p,
\eeq 
where $ \omega_n$ denotes the volume of the unit ball in $\brn$, 
and we have put 
\[
b(t):=e^{-n(p-1)Ht/2}  \{\omega_n r(t)^n\}^{-p+1}.
\]
From this, $w$ satisfies the differential inequality
\beq
\label{Thm-1-Proof-2500}
\partial_t^2w(t)+c^2Qw(t)-c^2b(t)w^p\ge0,
\eeq
which yields 
\[
\partial_t^2 w(t)+c^2Qw(t)\ge0 
\]
by $b(t) w(t)^p\ge0$.
Multiplying $\partial_t w$ to this inequality, integrating it, and using the assumptions $w_0\ge0$, $w_1\ge0$ and  $w_1\ge cM w_0$, we have 
$(\partial_t w)^2+c^2Qw^2\ge0$.
This inequality yields $\partial_t w-cMw\ge0$ by $w\ge0$ and $\partial_tw\ge0$, 
by which we obtain 
\beq
\label{Thm-1-Proof-3000} 
w(t) \ge w_0 e^{cMt}.
\eeq
We have $r(t)\le r_0+c/H$ for $t\ge0$ when $H>0$, 
and we also have $r(t)\le 2ce^{-Ht}/|H|$ for sufficiently large $t$ when $H<0$ 
by the definition of $r(t)$.
So that, $b(t)$ is bounded as 
\beq
\label{Thm-1-Proof-4000} 
b(t)\ge B e^{-n(p-1)|H|t/2}
\eeq
for $t\gg 1$, 
where $B$ is a constant defined by 
\[
B
:=
\omega_n^{-p+1} 
\begin{cases}
\left(r_0+\frac{c}{H}\right)^{-n(p-1)} & \mbox{if} \ H>0\ \mbox{and}\ t\ge0, 
\\
\left(\frac{2c}{|H|}\right)^{-n(p-1)} & \mbox{if} \ H<0\ \mbox{and}\ t\gg1.
\end{cases}
\]
In addition, 
we have 
\beq
\label{Thm-1-Proof-5000} 
b'(t)\le 0
\eeq
for $t\gg1$ 
since 
\[
b'(t)=-n(p-1) \omega_n^{-p+1} \left(e^{Ht/2} r(t)\right)^{-n(p-1)} 
\left(\frac{H}{2}+\frac{r'(t)}{r(t)}\right)
\]
and 
\[
\frac{r'(t)}{r(t)}=\frac{ce^{-Ht}}{r_0+\frac{c}{H}(1-e^{-Ht})}
\to 
\begin{cases}
0 & \mbox{if} \ H>0,\\
-H & \mbox{if} \ H<0
\end{cases}
\]
as $t\to \infty$.

Multiplying $\partial_tw$ to \eqref{Thm-1-Proof-2500}, 
which is non-negative,
we have 
\[
\partial_t e_0(t)+e_1(t)\ge0
\]
for $t\ge0$,
where we have put 
\[
e_0:=\frac{1}{2c^2} (\partial_t w)^2+\frac{Q}{2} w^2-\frac{b}{p+1} w^{p+1},
\ \ \ \ 
e_1:=\frac{\partial_t b}{p+1} w^{p+1}.
\]
Integrating the both sides of this inequality on the interval $[t_0, t]$ for sufficiently large $t_0\gg 1$, 
and using $b'\le 0$ in \eqref{Thm-1-Proof-5000}, 
we obtain 
\beq
\label{Thm-1-Proof-6000} 
\left(\partial_t w(t)\right)^2+c^2Qw(t)^2-\frac{2c^2b(t)}{p+1}w(t)^{p+1}\ge 2c^2 e_0(t_0).
\eeq
The term $bw^{p+1}$ in this inequality is estimated by 
\begin{eqnarray}
b(t)w(t)^{p+1}
&\ge& 
Bw_0^{p-1} e^{(p-1)(cM-n|H|/2)t} w(t)^2
\nonumber\\
&\ge& 
Bw_0^{p-1} w(t)^2
\label{Thm-1-Proof-7000} \\
&\ge&
Bw_0^{p+1} e^{2cMt}
\to \infty 
\nonumber 
\end{eqnarray}
as $t\to \infty$ by \eqref{Thm-1-Proof-3000}, \eqref{Thm-1-Proof-4000}, 
and $cM-n|H|/2\ge0$ due to $m_\ast\in i\br$.
Thus, the inequality \eqref{Thm-1-Proof-6000} yields 
\begin{eqnarray}
(\partial_t w(t))^2
&\ge&
-c^2Qw(t)^2+\frac{c^2 b(t)}{p+1} w(t)^{p+1}
\label{Thm-1-Proof-9000}
\\
&\ge&
c^2 M_1^2 w(t)^2
\nonumber
\end{eqnarray}
for $t\gg 1$ by \eqref{Thm-1-Proof-7000},
where we have put 
\[
M_1:=\left(M^2+\frac{Bw_0^{p-1}}{p+1}\right)^{1/2}.
\]
So that, we have $\partial_t w\ge cM_1 w$, 
by which we obtain 
\beq
\label{Thm-1-Proof-10000}
w(t)\ge w(t_1) e^{cM_1(t-t_1)}
\eeq
for  $t\ge t_1\gg1$.
For any sufficiently small number $\varepsilon>0$,
we estimate the term $b(t) w(t)^{p+1}$ in \eqref{Thm-1-Proof-9000} as  
\begin{eqnarray*}
b(t) w(t)^{p+1}
&=&
b(t)w(t)^{(1-\varepsilon)(p-1)} w(t)^{2+\varepsilon(p-1)}
\\
&\ge&
B\left(
\left(w(t_1)e^{-cM_1t_1}\right)^{1-\varepsilon} 
e^{-n|H|t/2+cM_1(1-\varepsilon)t}
\right)^{p-1} 
w(t)^{2+\varepsilon(p-1)}
\\
&\ge&
B
\left(w(t_1)e^{-cM_1t_1}\right)^{(1-\varepsilon)(p-1)} 
w(t)^{2+\varepsilon(p-1)}
\end{eqnarray*}
for $t\ge t_1\gg 1$ by \eqref{Thm-1-Proof-4000} and \eqref{Thm-1-Proof-10000}, 
where we have used 
$-n|H|t/2+cM_1(1-\varepsilon)t\ge0$ for sufficiently small $\varepsilon>0$ 
by $cM_1\ge cM> n|H|/2$ when $m_\ast\neq0$, 
and 
by 
$cM_1> cM= n|H|/2$ when $m_\ast=0$ and $w_0>0$.
By this estimate, \eqref{Thm-1-Proof-9000} and $Q\le0$, 
we obtain the differential inequality
\beq
\label{Thm-1-Proof-11000}
\partial_t w(t)
\ge 
c\sqrt{ \frac{B}{p+1} } 
\left(w(t_1)e^{-cM_1t_1}\right)^{(1-\varepsilon)(p-1)/2}
w(t)^{1+\varepsilon(p-1)/2}
\eeq
for $0<\varepsilon\ll 1$ and $t\ge t_1\gg 1$. 
Since $w$ is positive,  
and the positive solution of \eqref{Thm-1-Proof-11000} must blow up in finite time, 
the function $w$ blows up as required.

\newsection{Proof of Theorem \ref{Thm-1}}
\label{Section-Proof-Thm-1}
(1) 
Let $q=\infty$ when $H=0$, and $0\le 1/q\le \min\{1/2,1/(n-2\mu_0)\}$ when $H>0$.
Let $\tilde{q}=\infty$ when $H=0$, and $0\le 1/\tilde{q}\le \min\{1/2,2/(n-2\mu_0)\}$ when $H>0$.
Put $1/q_\ast:=1-(n-2\mu_0)/q$ and $1/\tilde{q}_\ast:=1-(n-2\mu_0)/\tilde{q}$.

Let $\rho_0$, $\rho_1$, $K_0$, $K_1$ and $K$ be the functions and the operators 
defined in Lemma \ref{Lem-7} and \eqref{Def-K} 
for the function  $\tilde{a}$ in \eqref{Def-aTilde}.
Then the solution of \eqref{Cauchy} is regarded as the fixed point of the operator $\Phi$ 
defined by 
\beq
\label{Def-Phi}
\Phi(u)(t)
:=
K_0(t) u_0+K_1(t) u_1-\int_0^t K(t,s) h(u)(s) ds,
\eeq
where $h(u)$ is defined by \eqref{Def-h}.
For constants $T>0$, $R_\nu>0$ for $\nu=0,\mu_0,\mu$, 
we define the closed ball defined by 
\beq
\label{Def-XTR}
X^\mu(T,R_0,R_{\mu_0}, R_\mu)
:=
\{u;\ \|u\|_{\dot{X}^\nu}\le R_\nu\ \ \mbox{for}\ \ \nu=0,\mu_0,\mu\},
\eeq
and we show $\Phi$ is a contraction mapping on this space 
for the suitable constants.
The solution is obtained as the fixed point of $\Phi$.
Let $\mu_0$ and $\mu$ satisfy \eqref{Assume-Mu}.
Let $q=\infty$ when $H=0$, $0\le 1/q\le \min\{1/2,1/(n-2\mu_0)\}$ when $H>0$.
We define $\theta$, $r_\ast$, $r_{\ast\ast}$, $q_\ast$ by 
\beq
\label{Proof-Thm-1-500}
\theta:=\frac{n-2\mu_0}{3},
\ \ 
\frac{1}{r_\ast}:=\frac{1}{6}-\frac{\mu_0}{3n},
\ \ 
\frac{1}{r_{\ast\ast}}:=\frac{1}{6}+\frac{2\mu_0}{3n},
\ \ 
\frac{1}{q_\ast}:=1-\frac{n-2\mu_0}{q}.
\eeq
We note 
\beq
\label{Proof-Thm-1-1000}
0\le \theta\le 1,
\ \ 
\frac{1}{2}=\frac{2}{r_\ast}+\frac{1}{r_{\ast\ast}},
\ \ 
0<\frac{1}{r_\ast}\le\frac{1}{2},
\ \ 
0<\frac{1}{r_{\ast\ast}}\le\frac{1}{2}
\eeq
hold by the definition of $\theta$, $(n-3)/2\le \mu_0<n/2$.
We have 
\beq
\label{Proof-Thm-1-1500}
\||u|^2u\|_{\dot{H}^\nu}
\lesssim
\|u\|_{L^{r_\ast}\cap \dot{B}^0_{r_\ast,2}}^2\|u\|_{\dot{B}^\nu_{r_{\ast\ast},2}}
\eeq
for $\nu\ge0$ 
by the nonlinear estimate in the Besov spaces 
(see \cite[Lemm 3.1]{Nakamura-Ozawa-2002-ASNSP}) 
and \eqref{Proof-Thm-1-1000}.
Since we have the embeddings  
$\dot{H}^{\mu_0+\theta}
\hookrightarrow 
\dot{B}^{\mu_0}_{r_{\ast\ast},2}
\hookrightarrow 
L^{r_\ast}\cap \dot{B}^0_{r_\ast,2}$ 
and 
$\dot{H}^{\nu+\theta}
\hookrightarrow 
\dot{B}^{\nu}_{r_{\ast\ast},2}$ 
by 
$1/r_\ast=1/r_{\ast\ast}-\mu_0/n$ 
and 
$1/r_{\ast\ast}=1/2-\theta/n$,
we have 
\begin{eqnarray}
\||u|^2u\|_{\dot{H}^\nu}
&\lesssim&
\|u\|_{\dot{H}^{\mu_0+\theta}}^2
\|u\|_{\dot{H}^{\nu+\theta} }
\nonumber\\
&\lesssim&
\|u\|_{\dot{H}^{\mu_0}}^{2(1-\theta)}
\|u\|_{\dot{H}^{\mu_0+1}}^{2\theta}
\|u\|_{\dot{H}^\nu}^{1-\theta}
\|u\|_{\dot{H}^{\nu+1}}^{\theta},
\label{Proof-Thm-1-1700}
\end{eqnarray}
where we have used the interpolation inequalities at the last line.
Thus, we obtain 
\begin{eqnarray}
\||u|^2ue^{-nHt}\|_{L^1\dot{H}^\nu}
&\lesssim& 
\|A\|_{q_\ast} 
\|u\|_{L^\infty \dot{H}^{\mu_0}}^{2(1-\theta)}
\|e^{-Ht}u\|_{L^q\dot{H}^{\mu_0+1}}^{2\theta}
\|u\|_{L^\infty \dot{H}^{\nu}}^{1-\theta}
\|e^{-Ht}u\|_{L^q\dot{H}^{\nu+1}}^{\theta}
\nonumber
\\
&\lesssim& 
\|A\|_{q_\ast} B\|u\|_{\dot{X}^{\mu_0}}^2\|u\|_{\dot{X}^\nu}
\label{Proof-Thm-1-2000}
\end{eqnarray}
for $\nu=0,\mu_0,\mu$ 
by the H\"older inequality and $1=1/q_\ast+3\theta/q$,
where we have put 
\beq
\label{Def-A-B}
A=A(t):=e^{-nHt+3\theta Ht} 
\ \ 
\mbox{and}
\ \  
B:=Q^{-3(1-\theta)/2} H^{-3\theta/q}.
\eeq
We note 
\beq
\label{Proof-Thm-1-3000}
\|A\|_{L^{q_\ast}((0,T))}
=
\begin{cases}
T^{1/q_\ast} 
& \mbox{if}\ \ 1\le q_\ast\le \infty, \ \ H\mu_0=0,
\\
1
& \mbox{if}\ \ q_\ast=\infty, \ \ H\mu_0>0,
\\
\left\{
\frac{1-e^{-2\mu_0HTq_\ast}}{2\mu_0Hq_\ast}
\right\}^{1/q_\ast}
& \mbox{if}\ \ q_\ast<\infty, \ \ H\mu_0>0
\end{cases}
\eeq
by a direct calculation.

Similarly to the above estimate, 
let $\tilde{q}=\infty$ when $H=0$, $0\le 1/\tilde{q}\le \min\{1/2,2/(n-2\mu_0)\}$ when $H>0$.
We define $\tilde{\theta}$, $\tilde{r}_\ast$, $\tilde{r}_{\ast\ast}$, $\tilde{q}_\ast$ by 
\beq
\label{Proof-Thm-1-3500}
\tilde{\theta}:=\frac{n-2\mu_0}{4},
\ \ 
\frac{1}{\tilde{r}_\ast}:=\frac{1}{4}-\frac{\mu_0}{2n},
\ \ 
\frac{1}{\tilde{r}_{\ast\ast}}:=\frac{1}{4}+\frac{\mu_0}{2n},
\ \ 
\frac{1}{\tilde{q}_\ast}:=1-\frac{n-2\mu_0}{2\tilde{q}}.
\eeq
We note 
\beq
\label{Proof-Thm-1-4000}
0\le \tilde{\theta}\le 1,
\ \ 
\frac{1}{2}=\frac{1}{\tilde{r}_\ast}+\frac{1}{\tilde{r}_{\ast\ast}},
\ \ 
0<\frac{1}{\tilde{r}_\ast}\le\frac{1}{2},
\ \ 
0<\frac{1}{\tilde{r}_\ast}\le\frac{1}{2}
\eeq
hold by the definition of $\tilde{\theta}$, $(n-4)/2\le \mu_0<n/2$.
We have 
\beq
\label{Proof-Thm-1-4500}
\|2u\real{u}+|u|^2\|_{\dot{H}^\nu}
\lesssim
\|u\|_{L^{\tilde{r}_\ast}\cap \dot{B}^0_{\tilde{r}_\ast,2}}\|u\|_{\dot{B}^\nu_{\tilde{r}_{\ast\ast},2}}
\eeq
by the nonlinear estimate 
and \eqref{Proof-Thm-1-4000}.
Since we have the embeddings  
$\dot{H}^{\mu_0+\tilde{\theta}}
\hookrightarrow 
\dot{B}^{\mu_0}_{\tilde{r}_{\ast\ast},2}
\hookrightarrow 
L^{\tilde{r}_\ast}\cap \dot{B}^0_{\tilde{r}_\ast,2}$ 
and 
$\dot{H}^{\nu+\tilde{\theta}}
\hookrightarrow 
\dot{B}^{\nu}_{\tilde{r}_{\ast\ast},2}$ 
by 
$1/\tilde{r}_\ast=1/\tilde{r}_{\ast\ast}-\mu_0/n$ 
and 
$1/\tilde{r}_{\ast\ast}=1/2-\tilde{\theta}/n$,
we have 
\begin{eqnarray}
\|2u\real{u} +|u|^2\|_{\dot{H}^\nu}
&\lesssim&
\|u\|_{\dot{H}^{\mu_0+\tilde{\theta}}}
\|u\|_{\dot{H}^{\nu+\tilde{\theta}} }
\nonumber
\\
&\lesssim&
\|u\|_{\dot{H}^{\mu_0}}^{1-\tilde{\theta}}
\|u\|_{\dot{H}^{\mu_0+1}}^{\tilde{\theta}}
\|u\|_{\dot{H}^\nu}^{1-\tilde{\theta}}
\|u\|_{\dot{H}^{\nu+1}}^{\tilde{\theta}},
\label{Proof-Thm-1-4700}
\end{eqnarray}
where we have used the interpolation inequalities at the last line.
Thus, we obtain 
\begin{eqnarray}
&&
\|(2u\real u+|u|^2)e^{-nHt/2}\|_{L^1\dot{H}^\nu}
\nonumber
\\
&\lesssim& 
\|\tilde{A}\|_{\tilde{q}_\ast} 
\|u\|_{ L^\infty \dot{H}^{\mu_0} }^{ 1-\tilde{\theta} }
\|e^{-Ht}u\|_{ L^{\tilde{q}} \dot{H}^{\mu_0+1} }^{ \tilde{\theta} }
\|u\|_{L^\infty \dot{H}^{\nu}}^{1-\tilde{\theta}}
\|e^{-Ht}u\|_{L^{\tilde{q}}\dot{H}^{\nu+1}}^{\tilde{\theta}}
\nonumber
\\
&\lesssim& 
\|\tilde{A}\|_{\tilde{q}_\ast} \tilde{B}\|u\|_{\dot{X}^{\mu_0}}\|u\|_{\dot{X}^\nu}
\label{Proof-Thm-1-5000}
\end{eqnarray}
for $\mu=0,\mu_0,\mu$ 
by the H\"older inequality and $1=1/\tilde{q}_\ast+2\tilde{\theta}/\tilde{q}$,
where we have put 
$\tilde{A}=\tilde{A}(t):=e^{-nHt/2+2\tilde{\theta} Ht}$ 
and 
$\tilde{B}=Q^{-(1-\tilde{\theta})} H^{-2\tilde{\theta}/\tilde{q}}$.
We note 
\beq
\label{Proof-Thm-1-6000}
\|\tilde{A}\|_{L^{\tilde{q}_\ast}((0,T))}
=
\begin{cases}
T^{1/\tilde{q}_\ast} 
& \mbox{if}\ \ 1\le \tilde{q}_\ast\le \infty, \ \ H\mu_0=0,
\\
1
& \mbox{if}\ \ \tilde{q}_\ast=\infty, \ \ H\mu_0>0,
\\
\left\{
\frac{1-e^{-\mu_0HT\tilde{q}_\ast}}{\mu_0H\tilde{q}_\ast}
\right\}^{1/\tilde{q}_\ast}
& \mbox{if}\ \ \tilde{q}_\ast<\infty, \ \ H\mu_0>0
\end{cases}
\eeq
by a direct calculation.

By \eqref{Proof-Thm-1-2000} and \eqref{Proof-Thm-1-5000}, 
we have 
\beq
\label{Proof-Thm-1-7000}
\|h(u)\|_{L^1\dot{H}^\nu}
\lesssim
\lambda\|A\|_{q_\ast}BR_{\mu_0}^2R_\nu
+
\lambda r_0\|\tilde{A}\|_{\tilde{q}_\ast}\tilde{B}R_{\mu_0}R_\nu
\eeq
for $\nu=0,\mu_0,\mu$, and any $u\in X^\mu(T,R_0,R_{\mu_0},R_\mu)$.
Since we have 
\[
\|\Phi(u)\|_{\dot{X}^\nu}
\lesssim 
\frac{1}{c}\|u_1\|_{\dot{H}^\nu}
+
\|\nabla u_0\|_{\dot{H}^\nu}
+
\sqrt{Q}\|u_0\|_{\dot{H}^\nu}
+
c\|h(u)\|_{L^1\dot{H}^\nu}
\]
by Lemma \ref{Lem-4}, 
we have 
\beq
\label{Proof-Thm-1-8000}
\|\Phi(u)\|_{\dot{X}^\nu}
\le 
C_0\dot{D}^\nu
+
Cc\lambda 
\left(
\|A\|_{q_\ast}BR_{\mu_0}
+
r_0\|\tilde{A}\|_{\tilde{q}_\ast}\tilde{B}
\right)
R_{\mu_0}R_\nu
\le R_\nu
\eeq
for $\nu=0,\mu_0,\mu$ 
for some constants $C_0>0$, $C>0$, and any $u\in X^\mu(T,R_0,R_{\mu_0},R_\mu)$ 
by \eqref{Proof-Thm-1-7000} 
if $R_0$, $R_{\mu_0}$ and $R_\mu$ satisfy 
\beq
\label{Proof-Thm-1-9000}
R_\nu\ge 2C_0 \dot{D}^\nu,
\ \ 
Cc\lambda 
\left(
\|A\|_{q_\ast}BR_{\mu_0}
+
r_0\|\tilde{A}\|_{\tilde{q}_\ast}\tilde{B}
\right)
R_{\mu_0}
\le \frac{1}{2}
\eeq
for $\nu=0,\mu_0,\mu$.

Next, we consider the estimate for the metric.
Since we have 
\[
\||u|^2u-|v|^2v\|_2
\lesssim
\max_{w=u,v}\|w\|_{r_\ast}^2\|u-v\|_{r_{\ast\ast}}
\]
for any functions $u$ and $v$ similarly to \eqref{Proof-Thm-1-1500} 
by the H\"older inequality,
we obtain 
\[
\|(|u|^2u-|v|^2v)e^{-nHt}\|_{L^1 L^2}
\lesssim 
\|A\|_{q_\ast} B\max_{w=u,v}\|w\|_{\dot{X}^{\mu_0}}^2\|u-v\|_{\dot{X}^0}
\]
similarly to \eqref{Proof-Thm-1-2000} by the same argument. 
Since we also have 
\[
\|2u\real{u}+|u|^2-(2v\real{v}+|v|^2)\|_2
\lesssim
\max_{w=u,v}
\|w\|_{ L^{\tilde{r}_\ast} } \|u-v\|_{ L^{ \tilde{r}_{\ast\ast} } }
\]
similarly to \eqref{Proof-Thm-1-4500}, 
we obtain 
\[
\|\{2u\real{u}+|u|^2-(2v\real{v}+|v|^2)\}e^{-nHt/2}\|_{L^1 L^2}
\lesssim 
\|\tilde{A}\|_{\tilde{q}_\ast} \tilde{B}\max_{w=u,v}\|w\|_{\dot{X}^{\mu_0}}\|u-v\|_{\dot{X}^0}
\]
similarly to \eqref{Proof-Thm-1-5000}.
So that, we have 
\begin{eqnarray*}
&& 
\|h(u)-h(v)\|_{L^1 L^2}
\\
&\lesssim&
\lambda\|A\|_{q_\ast}B
\max_{w=u,v}\|w\|_{\dot{X}^{\mu_0}}^2
\|u-v\|_{\dot{X}^0}
+
\lambda r_0\|\tilde{A}\|_{\tilde{q}_\ast}\tilde{B}
\max_{w=u,v}\|w\|_{\dot{X}^{\mu_0}}
\|u-v\|_{\dot{X}^0}
\\
&\lesssim&
\lambda\|A\|_{q_\ast}BR_{\mu_0}^2\|u-v\|_{\dot{X}^0}
+
\lambda r_0\|\tilde{A}\|_{\tilde{q}_\ast}\tilde{B}R_{\mu_0}\|u-v\|_{\dot{X}^0}
\end{eqnarray*}
for any $u,v\in X^\mu(T,R_0,R_{\mu_0},R_\mu)$ 
similarly to \eqref{Proof-Thm-1-7000},
by which we obtain 
\beq
\label{Proof-Thm-1-10000}
d(\Phi(u),\Phi(v))
\le 
Cc\lambda 
\left(
\|A\|_{q_\ast}BR_{\mu_0}
+
r_0\|\tilde{A}\|_{\tilde{q}_\ast}\tilde{B}
\right)
R_{\mu_0}d(u,v)
\le \frac{1}{2} \, d(u,v)
\eeq
by Lemma \ref{Lem-4} 
similarly to \eqref{Proof-Thm-1-8000} 
under the second condition in \eqref{Proof-Thm-1-9000}. 

We take 
$q=\tilde{q}=\infty$, thus, $q_\ast=\tilde{q}_\ast=1$, 
and 
$R_{\mu_0}=2C_0 \dot{D}^{\mu_0}$. 
Then the second condition in \eqref{Proof-Thm-1-9000} is satisfied 
for sufficiently small $T>0$, 
and $T$ depends on the size of $\dot{D}^{\mu_0}$.
So that, $\Phi$ is a contraction mapping, and we obtain the local in time solutions.

The continuity of the solution $u\in C([0,T),H^{\mu+1})\cap C^1([0,T),H^\mu)$ 
follows from the continuity of the operators $K_0$, $K_1$ and $K$ 
such as 
$K_0(\cdot)u_0\in C(\br,H^{\mu+1})$,
$\partial_t K_0(\cdot)u_0\in C(\br,H^{\mu})$ 
for $u_0\in H^{\mu+1}$,
$K_1(\cdot)u_1\in C(\br,H^{\mu+1})$,
$\partial_t K_1(\cdot)u_1\in C(\br,H^{\mu})$ 
for $u_1\in H^{\mu}$,
$\int_0^tK(t,s)h(s)ds\in C(\br_t,H^{\mu+1})$,
$\partial_t\int_0^tK(t,s)h(s)ds\in C(\br_t,H^{\mu})$ 
for $h\in L^1H^\mu$.
The uniqueness of the solution in 
$C([0,T),H^{\mu+1})\cap C^1([0,T),H^\mu)\cap X^\mu(T)$ 
follows from the continuity of the solution, 
and the result that the existence time $T$ is 
taken on the size of the norm of the data in our argument.
See e.g., \cite{Nakamura-2021-JDE} for the details.

(2) 
Let $v$ be the solution of the Cauchy problem for the data $v_0$ and $v_1$.
Put 
$\dot{D}^0(u-v)
:=
\frac{1}{c}\|u_1-v_1\|_2
+
\|\nabla (u_0-v_0)\|_2
+
\sqrt{Q}\|u_0-v_0\|_2$.
By Lemma \ref{Lem-4} and the similar argument to derive \eqref{Proof-Thm-1-8000}, 
we have 
\[
d(u,v)
\lesssim
\dot{D}^0(u-v)
+
c\|h(u)-h(v)\|_{L^1L^2}
\]
and thus, 
\[
d(u,v)
\le
C_0\dot{D}^0(u-v)
+
Cc\lambda 
\left(
\|A\|_{q_\ast}BR_{\mu_0}
+
r_0\|\tilde{A}\|_{\tilde{q}_\ast}\tilde{B}
\right)
R_{\mu_0}d(u,v).
\]
Since $u$ is the solution in $X^\mu(T,R_0,R_{\mu_0},R_\mu)$ 
under the condition \eqref{Proof-Thm-1-9000},
$v$ is in $X^\mu(T,R_0+\varepsilon,R_{\mu_0}+\varepsilon,R_\mu+\varepsilon)$ 
for sufficiently small $\varepsilon>0$ 
when $(v_0,v_1)$ is sufficiently close to $(u_0,u_1)$.
So that, 
$Cc\lambda 
\left(
\|A\|_{q_\ast}BR_{\mu_0}
+
r_0\|\tilde{A}\|_{\tilde{q}_\ast}\tilde{B}
\right)
R_{\mu_0}<1$ 
when $(v_0,v_1)$ is sufficiently close to $(u_0,u_1)$, 
which yields 
$d(u,v)\to 0$ as $(v_0,v_1)\to(u_0,u_1)$.

(3)  
We take $q_\ast=\tilde{q}_\ast=\infty$, 
thus, 
$q=n$, $\tilde{q}=n/2$ for the condition (i).
We take 
$q_\ast<\infty$, $\tilde{q}_\ast<\infty$ when 
for the condition (ii).
Then the second condition in \eqref{Proof-Thm-1-9000} is satisfied by 
\beq
\label{Proof-Thm-1-T-2000}
Cc\lambda(BR_{\mu_0}+r_0\tilde{B})R_{\mu_0}\le \frac{1}{2}
\eeq
for (i), 
where we need at least $n\ge4$ to make $\tilde{q}_\ast=\infty$ by $\tilde{q}=n/2\ge2$,
or 
\beq
\label{Proof-Thm-1-T-3000}
Cc\lambda
\left(
(2\mu_0Hq_\ast)^{-1/q_\ast}BR_{\mu_0}
+r_0(\mu_0H\tilde{q}_\ast)^{-1/\tilde{q}_\ast}\tilde{B}
\right)
R_{\mu_0}\le \frac{1}{2}.
\eeq
for (ii).
Since \eqref{Proof-Thm-1-T-2000} or \eqref{Proof-Thm-1-T-3000} 
holds for $T=\infty$ and sufficiently small $\dot{D}^{\mu_0}>0$, 
We obtain the global solutions under the condition (i) or (ii).

(4) 
The required results follow directly from Lemma \ref{Lem-10} and $h(u)\in L^1((0,\infty),H^\mu)$ as we have shown in \eqref{Proof-Thm-1-7000}.

(5) 
The local in time solution is obtained in (1) by setting $\mu_0=\mu=0$. 
Integrating the both sides in the equation 
$\partial_0{\tilde{e}}^0+\partial_j e^j+\tilde{e}^{n+1}=0$ 
in (3) in Lemma \ref{Lem-4}, 
we have 
\[
\int_\brn \tilde{e}^0(t) dx
=
\frac{1}{2c^2}\|\partial_t u(t)\|_2^2+\frac{1}{2}\|\nabla u(t)\|_2^2
+
\lambda\left\|r_0\real u(t)+\frac{|u(t)|^2}{2}\right\|_{L^2}^2
=
\int_\brn \tilde{e}^0(0) dx
\]
by the divergence theorem, 
where we have used 
\[
\tilde{e}^0
=
\frac{1}{2c^2}|\partial_t u|^2
+\frac{1}{2} |\nabla u|^2
+
\lambda
\left(
r_0\real u+\frac{1}{2}|u|^2
\right)^2
\]
and $\tilde{e}^{n+1}=0$ by $H=0$.
So that, 
$\|\partial_t u(t)\|_2$ and 
$\|\nabla u(t)\|_2$ are uniformly bounded.
In addition, 
$\|u(t)\|_2$ does not blow up 
since 
$\|\partial_t u(t)\|_2$ is bounded by 
\[
\|u(t)\|_2\le \|u(0)\|_2+\int_0^t\|\partial_t u(s)\|_2 ds.
\]
Since the existence time of our solutions obtained in (1) is taken by the size of the norm of the data 
$\dot{D}^0$, 
we are able to show the existence of the global solution connecting the local solution.

%

\newsection{Proof of Theorem \ref{Thm-2}}
\label{Section-Proof-Thm-2}
The proof of Theorem \ref{Thm-2} follows analogously to that of Theorem \ref{Thm-1}.
We only focus on the essential parts to prove (1).

(1) 
We consider the operator $\Phi$ 
defined by \eqref{Def-Phi}, 
and we show that $\Phi$ is a contraction mapping on 
the closed ball defined by \eqref{Def-XTR} 
for some $T>0$, $R_\nu>0$, $\nu=0,\mu_0,\mu$, 
where $\|\cdot\|_{\dot{X}^\nu}$ is defined by \eqref{Thm-2-X}.
We define $\theta$, $r_\ast$, $r_{\ast\ast}$ by \eqref{Proof-Thm-1-500}.
Since we have the property \eqref{Proof-Thm-1-1000}, 
we obtain the estimates \eqref{Proof-Thm-1-1500} and \eqref{Proof-Thm-1-1700} by the same argument.
For any $q_0$ with $2\le q_0\le \infty$, 
assume $3-n+2\mu_0\le q_0$ when $\mu_0>(n-3)/2$.
For any $q$ with 
\[
0\le \frac{1}{q}\le \min\left\{\frac{1}{2},1-\frac{3-n+2\mu_0}{q_0} \right\},
\]
we define $q_\ast$ by 
\beq
\label{Def-qAst-Thm-2}
\frac{1}{q_\ast}:=\frac{1}{q'}-\frac{3-n+2\mu_0}{q_0}.
\eeq
We note that $q_\ast$ satisfies $1/q'=1/q_\ast+3(1-\theta)/q_0$, 
and 
$q'\le q_\ast\le \infty$ holds by the conditions on $q_0$ and $q$.
Thus, we have 
\begin{eqnarray}
&&
c(-H)^{-1/q}\|e^{-(n-1)Ht} |u|^2u\|_{L^{q'} \dot{H}^\nu}
\nonumber
\\
&\lesssim&
B\|A\|_{q_\ast}
\left\{
(-H)^{1/q_0}Q^{1/2} 
\|e^{Ht}u\|_{L^{q_0}\dot{H}^{\mu_0}}
\right\}^{2(1-\theta)} 
\|u\|_{L^\infty \dot{H}^{\mu_0+1}}^{2\theta}
\nonumber\\
&&
\ \ \ \ 
\cdot
\left\{
(-H)^{1/q_0}Q^{1/2} 
\|e^{Ht}u\|_{L^{q_0}\dot{H}^{\nu}}
\right\}^{1-\theta} 
\|u\|_{L^\infty \dot{H}^{\nu+1}}^{\theta }
\nonumber
\\
&\le& 
B\|A\|_{q_\ast}
\|u\|_{\dot{X}^{\mu_0}}^{2} \|u\|_{\dot{X}^{\nu}} 
\nonumber
\\
&\le& 
B\|A\|_{q_\ast}R_{\mu_0}^{2} R_\nu
\label{Proof-Thm-2-3000}
\end{eqnarray}
for $\nu=0,\mu_0,\mu$, and any $u\in X^\mu(T,R_0,R_{\mu_0},R_\mu)$,
where we have put 
\beq
A=A(t):=e^{-2(1+\mu_0) Ht }
\ \ 
\mbox{and}
\ \ 
B:=c(-H)^{-1+1/q_\ast} Q^{(n-3-2\mu_0)/2}.
\eeq
We note 
\[
\|A\|_{L^{q_\ast}((0,T))}
=
\begin{cases}
e^{-2(1+\mu_0)HT} & \mbox{if}\ \ q_\ast=\infty,
\\
\left\{
-\frac{e^{-2(1+\mu_0)Hq_\ast T}-1}{2(1+\mu_0)Hq_\ast}
\right\}^{1/q_\ast}
& \mbox{if}\ \ q_\ast<\infty.
\end{cases}
\]

We define $\tilde{\theta}$, $\tilde{r}_\ast$, $\tilde{r}_{\ast\ast}$ by \eqref{Proof-Thm-1-3500}.
Since we have the property \eqref{Proof-Thm-1-4000}, 
we obtain the estimates \eqref{Proof-Thm-1-4500} and \eqref{Proof-Thm-1-4700} by the same argument.
For any $\tilde{q}_0$ with $2\le \tilde{q}_0\le \infty$, 
assume 
$(4-n+2\mu_0)/2\le \tilde{q}_0$ when $\mu_0>(n-4)/2$.
For any $\tilde{q}$ with 
\[
0\le \frac{1}{\tilde{q}}\le \min\left\{\frac{1}{2},1-\frac{4-n+2\mu_0}{2\tilde{q}_0} \right\},
\]
we define $\tilde{q}_\ast$ by 
\beq
\label{Def-TilQAst-Thm-2}
\frac{1}{\tilde{q}_\ast}:=\frac{1}{\tilde{q}'}-\frac{4-n+2\mu_0}{2\tilde{q}_0}.
\eeq
We note that $\tilde{q}_\ast$ satisfies $1/\tilde{q}'=1/\tilde{q}_\ast+2(1-\tilde{\theta})/\tilde{q}_0$, 
and 
$\tilde{q}'\le \tilde{q}_\ast\le \infty$ holds by the conditions on $\tilde{q}_0$ and $\tilde{q}$.
Thus, we have 
\begin{eqnarray}
&&
c(-H)^{-1/\tilde{q}}\|e^{-(n-2)Ht/2} (2u\real u+|u|^2)\|_{L^{\tilde{q}'} \dot{H}^\nu}
\nonumber
\\
&\lesssim&
\tilde{B}\|\tilde{A}\|_{\tilde{q}_\ast}
\left\{
(-H)^{1/\tilde{q}_0}Q^{1/2} 
\|e^{Ht}u\|_{L^{\tilde{q}_0}\dot{H}^{\mu_0}}
\right\}^{1-\tilde{\theta}} 
\|u\|_{L^\infty \dot{H}^{\mu_0+1}}^{\tilde{\theta}}
\nonumber\\
&&
\ \ \ \ 
\cdot
\left\{
(-H)^{1/\tilde{q}_0}Q^{1/2} 
\|e^{Ht}u\|_{L^{\tilde{q}_0}\dot{H}^{\nu}}
\right\}^{1-\tilde{\theta}} 
\|u\|_{L^\infty \dot{H}^{\nu+1}}^{\tilde{\theta} }
\nonumber
\\
&\le& 
\tilde{B}\|\tilde{A}\|_{\tilde{q}_\ast}
\|u\|_{\dot{X}^{\mu_0}} \|u\|_{\dot{X}^{\nu}} 
\nonumber
\\
&\le& 
\tilde{B}\|\tilde{A}\|_{\tilde{q}_\ast}R_{\mu_0} R_\nu
\label{Proof-Thm-2-4000}
\end{eqnarray}
for $\nu=0,\mu_0,\mu$, and any $u\in X^\mu(T,R_0,R_{\mu_0},R_\mu)$,
where we have put 
\beq
\tilde{A}=\tilde{A}(t):=e^{-(1+\mu_0) Ht }
\ \ 
\mbox{and}
\ \ 
\tilde{B}:=c(-H)^{-1+1/\tilde{q}_\ast} Q^{(n-4-2\mu_0)/4}.
\eeq
We note 
\[
\|\tilde{A}\|_{L^{\tilde{q}_\ast}((0,T))}
=
\begin{cases}
e^{-(1+\mu_0)HT} & \mbox{if}\ \ \tilde{q}_\ast=\infty,
\\
\left\{
-\frac{e^{-(1+\mu_0)H\tilde{q}_\ast T}-1}{(1+\mu_0)H\tilde{q}_\ast}
\right\}^{1/\tilde{q}_\ast}
& \mbox{if}\ \ \tilde{q}_\ast<\infty.
\end{cases}
\]

Since we have 
\begin{multline*}
\|\Phi(u)\|_{\dot{X}^\nu}
\lesssim 
\frac{1}{c}\|u_1\|_{\dot{H}^\nu}
+
\|\nabla u_0\|_{\dot{H}^\nu}
+
\sqrt{Q}\|u_0\|_{\dot{H}^\nu}
\\
+
\lambda c(-H)^{-1/q}\|e^{-(n-1)Ht} |u|^2 u\|_{L^{q'}\dot{H}^\nu}
\\
+
\lambda r_0 c(-H)^{-1/\tilde{q} }\|e^{-(n-2)Ht/2} (2u\real u+|u|^2) \|_{L^{\tilde{q}'}\dot{H}^\nu}
\end{multline*}
by Lemma \ref{Lem-5}, 
we have 
\beq
\label{Proof-Thm-2-8000}
\|\Phi(u)\|_{\dot{X}^\nu}
\le 
C_0\dot{D}^\nu
+
C\lambda 
\left(
\|A\|_{q_\ast}BR_{\mu_0}
+
r_0\|\tilde{A}\|_{\tilde{q}_\ast}\tilde{B}
\right)
R_{\mu_0}R_\nu
\le R_\nu
\eeq
for $\nu=0,\mu_0,\mu$ 
for some constants $C_0>0$, $C>0$, 
and any $u\in X^\mu(T,R_0,R_{\mu_0},R_\mu)$ 
by \eqref{Proof-Thm-2-3000} and \eqref{Proof-Thm-2-4000} 
if $R_0$, $R_{\mu_0}$ and $R_\mu$ satisfy 
\beq
\label{Proof-Thm-2-9000}
R_\nu\ge 2C_0 \dot{D}^\nu,
\ \ 
C\lambda 
\left(
\|A\|_{q_\ast}BR_{\mu_0}
+
r_0\|\tilde{A}\|_{\tilde{q}_\ast}\tilde{B}
\right)
R_{\mu_0}
\le \frac{1}{2}
\eeq
for $\nu=0,\mu_0,\mu$.
On the metric, we are able to obtain 
\beq
\label{Proof-Thm-2-10000}
d(\Phi(u),\Phi(v))
\le 
C\lambda 
\left(
\|A\|_{q_\ast}BR_{\mu_0}
+
r_0\|\tilde{A}\|_{\tilde{q}_\ast}\tilde{B}
\right)
R_{\mu_0}d(u,v)
\le \frac{1}{2} \, d(u,v)
\eeq
for any $u,v\in X(T,R_0,R_{\mu_0},R_\mu)$ analogously to \eqref{Proof-Thm-1-10000}, 
provided the second condition in \eqref{Proof-Thm-2-9000}.
So that, $\Phi$ is a contraction mapping on $X(T,R_0,R_{\mu_0},R_\mu)$ 
under  \eqref{Proof-Thm-2-9000}.
Especially, \eqref{Proof-Thm-2-9000} holds if $T>0$ is sufficiently small such that 
\begin{multline*}
C\lambda c(-H)^{-1}
\left[
\left\{
\frac{e^{-4(1+\mu_0)HT}-1}{4(1+\mu_0)}
\right\}^{1/2}
Q^{(n-3-2\mu_0)/2} R_{\mu_0}
\right.\\
+
\left.
r_0
\left\{
\frac{e^{-2(1+\mu_0)HT}-1}{2(1+\mu_0)}
\right\}^{1/2}
Q^{(n-4-2\mu_0)/4}
\right]
\le \frac{1}{2}
\end{multline*}
when $q_0=\tilde{q}_0=\infty$ and $q=\tilde{q}=q_\ast=\tilde{q}_\ast=2$, 
and 
$R_\nu=2C_0\dot{D}^\nu$ for $\nu=0,\mu_0,\mu$.

\vspace{10pt}

%

{\bf Acknowledgment.}
This work was supported by JSPS KAKENHI Grant Number 16H03940.

%

\end{document}